\documentclass{amsart}
\usepackage{amsmath,amsthm}
\usepackage{amsfonts,amssymb}
\usepackage{accents}
\usepackage{enumerate}
\usepackage{accents,color}
\usepackage{graphicx}

\newcommand{\lvt}{\left|\kern-1.35pt\left|\kern-1.3pt\left|}
\newcommand{\rvt}{\right|\kern-1.3pt\right|\kern-1.35pt\right|}

\newtheorem{thm}{Theorem}[section]
\newtheorem{cor}[thm]{Corollary}
\newtheorem{lem}[thm]{Lemma}
\newtheorem{prop}[thm]{Proposition}

\theoremstyle{remark}

 \def\la{{\langle}}
 \def\ra{{\rangle}}
 \def\ve{{\varepsilon}}

 \def\d{\mathrm{d}}

 \def\sd{{\mathsf d}}
 \def\sw{{\mathsf w}}
 \def\sB{{\mathsf B}}

 \def\a{{\alpha}}
 \def\b{{\beta}}
 \def\g{{\gamma}}
 \def\k{{\kappa}}
 \def\t{{\theta}}
 \def\l{{\lambda}}

 \def\la{{\langle}}
 \def\ra{{\rangle}}
 \def\ve{{\varepsilon}}

 \def\kb{{\mathbf k}}

 \def\CD{{\mathcal D}}

 \def\CV{{\mathcal V}}

 \def\BB{{\mathbb B}}

 \def\NN{{\mathbb N}}

 \def\RR{{\mathbb R}}

      \def\proj{\operatorname{proj}}

\def\lla{\langle{\kern-2.5pt}\langle}      
\def\rra{\rangle{\kern-2.5pt}\rangle}

\def\bk{{\boldsymbol{\kappa}}}

\newcommand{\wh}{\widehat}

\def\f{\frac}

\graphicspath{{./}}
\begin{document}

\title{Sharp Bernstein inequalities on simplex}
\author{Yan Ge}
\address{Department of Mathematics, University of Oregon, Eugene, 
OR 97403--1222, USA}
\email{yge7@uoregon.edu} 
\author{Yuan~Xu}
\address{Department of Mathematics, University of Oregon, Eugene, 
OR 97403--1222, USA}
\email{yuan@uoregon.edu} 
\thanks{The second author was partially supported by Simons Foundation Grant \#849676.}
\date{\today}  
.

\begin{abstract}  
We prove several new families of Bernstein inequalities of two types on the simplex. The first type consists of
inequalities in $L^2$ norm for the Jacobi weight, some of which are sharp, and they are established via the 
spectral operator that has orthogonal polynomials as eigenfunctions. The second type consists of inequalities 
in $L^p$ norm for doubling weight on the simplex. The first type is not necessarily a special case of the
second type when $d \ge 3$. 
\end{abstract}

\maketitle
 
\section{Introduction}
\setcounter{equation}{0}

We revisit Bernstein inequalities on the simplex in order to understand a new family of Bernstein
inequalities on the triangle that appeared recently in \cite{X23}. Let $\triangle^d$ be the simplex defined by 
$$
  \triangle^d := \left\{x \in \RR^d: x_1 \ge 0, \ldots, x_d \ge 0, \, |x| \le 1 \right\}, \quad |x|:= x_1+\cdots + x_d. 
$$
For $\bk = (\k_1,\ldots, \k_{d+1})$, let $W_\bk$ be the Jacobi weight function on the simplex,
$$
  W_\bk(x) := x_1^{\k_1} \cdots x_d^{\k_d} (1-|x|)^{\k_{d+1}}, \qquad \k_i > -1, \quad i =1, 2, \ldots, d+1.
$$
Denote by $\partial_j$ the $j$-th partial derivative. If $f$ is a polynomials of degree at most $n$ on $\RR^d$, 
then we have the Bernstein inequality \cite{BX, Dit, DT}:  for $1 \le p \le \infty$ and $1\le j \le d$,
\begin{equation} \label{eq:oldBI}
    \| \psi_j \partial_j f \|_{\bk,p}  \le c\,n \|f\|_{\bk,p} \quad\hbox{with}\quad \psi_j (x) =\sqrt{x_j (1-|x|)},  
\end{equation}
where $\|\cdot\|_{\bk,p}$ denotes the $L^p(\triangle^d, W_\bk)$ norm for $1\le p < \infty$ and uniform norm
if $p=\infty$, which agrees with the classical Bernstein inequality on $[0,1]$ when $d =1$. The factor $\psi_j$ 
in front of $\partial_j$ is a trade mark of Bernstein inequalities; see, for example, \cite{Br, DP, DT, DX, K2009, X23}. 
Recently, a new set of inequalities on the triangle, $\triangle^d$ when $d =2$, appeared in 
\cite{X23}, which shows in particular that, for $i, j = 1,2$ and $i \ne j$,  
\begin{align} \label{eq:B-ineqTri}
 \left \|  \frac{1}{\sqrt{1-x_j} } \psi_i \partial_i f \right \|_{p,w} \le c\, n \|f\|_{p,w},
\end{align}
which is evidently stronger than that of \eqref{eq:oldBI}. The inequality \eqref{eq:B-ineqTri} came out as 
a byproduct in \cite{X23}, derived from the Bernstein inequalities on the conic domains. It raises several
questions. 

One obvious question is a possible extension of \eqref{eq:B-ineqTri} to the simplex $\triangle^d$. The answer
is positive but it turns out, somewhat surprisingly, that there are several families of extensions. Another question
is a possible geometric interpretation of the new factor in front of the derivative. In a recent paper, Kro\'o 
\cite{K2023} establishes the Bernstein inequality on polyhedra, including the simplex as a special case, with $\psi_j$ 
bing the Euclidean distance from $x$ to the boundary of the domain. As we shall show later, however, our $\psi_{j}$ 
has a geometric interpretation but it is not via Euclidean distance. One more question arises from the link 
between the Bernstein inequality in $L^2$ norm and the spectral operator, see \cite{K2022, X23}, where the latter 
is a second-order linear differential operator on the domain that has orthogonal polynomials as eigenfunctions. 
For the simplex, the spectral operator is given by \cite[section 5.3]{DX}
\begin{align}\label{eq:Dk}
  \CD_\bk =  \sum_{i=1}^d x_i(1-x_i) \frac{\partial}{ \partial x_i^2}  
    & - 2 \sum_{1 \le i< j \le d} x_i x_j  \frac{\partial}{\partial x_i x_j} \\
    &  + \sum_{i=1}^d \big(\k_i +1 - (|\bk|+d+1) x_i \big) \frac{\partial}{\partial x_i} \notag
\end{align} 
and the Bernstein inequality \eqref{eq:oldBI} in the $L^2$ norm can be derived from the self-adjoint integration 
of $\CD_\bk$, in which $\phi_j \partial_j$ appears (see \eqref{eq:self-adj} below). One immediate question is if 
and how the new inequality \eqref{eq:B-ineqTri} is related to the spectral operator. 

The answer to the last question is somewhat surprising: it turns out that $\CD_\bk$ can be written in several other 
symmetric forms that lead to self-adjoint 
integrations of $\CD_\bk$. Each of the new symmetric forms contains fractions such as the one in 
front of $\partial_j$ in \eqref{eq:B-ineqTri} and each yields a family of new Bernstein inequalities
in $L^2$ norm on the simplex, of which some are sharp in the sense that equality holds for some 
polynomials. For the first two questions, we are able to establish new Bernstein inequalities for all 
derivatives in $L^p$ norm that extend \eqref{eq:B-ineqTri} to $\triangle^d$. They agree with those
in $L^2$ norm derived from the spectral operator when $d = 2$ but, surprisingly, not for $d > 2$. 
Moreover, these $L^p$ inequalities are established for all doubling weight on the simplex, not just 
for $W_\bk$. The proof follows the approach in \cite{X23} that relies on a framework based on highly
localized kernels developed in \cite{X21} for localizable space of homogeneous type. As a byproduct, 
our proof shows that $\triangle^d$ with $W_\bk$ and its intrinsic distance is such a space, which is
to be expected but not verified before.  

The paper is organized as follows. In the next section, we derive Bernstein inequalities in $L^2$ norm
from the spectral operator, which requires only a basic definition of orthogonality, and discuss their 
properties. The Bernstein inequalities in $L^p$ norm will be stated and discussed in the first 
subsection of Section 3 and proved in later subsections. 

Throughout this paper, we adopt the convention that the letter $c$ denotes a constant that depends
on fixed parameters, such as $d$, $p$ and $\bk$ but independent of the degree $n$ of polynomials.

\section{Spectral operator and $L^2$ Bernstein inequality}
\setcounter{equation}{0}

In this section, we prove the Bernstein inequality in the $L^2$ norm via the Jacobi weight function 
on the simplex. The triangle $\triangle^d$  is symmetric under permutations of $(x_1,\ldots, x_d, x_{d+1})$. 
We often write $x_{d+1} = 1- |x|$, so that the Jacobi weight function can be written as 
$$
W_{\bk}(x) = \prod_{i=1}^{d} x_i^{\k_i}(1-|x|)^{\k_{d+1}} =  \prod_{i=1}^{d+1} x_i^{\k_i}.
$$
The classical orthogonal polynomials on the simplex are orthogonal with respect to the inner product 
$$
  \la f, g \ra_{\bk} = b_\bk \int_{\triangle^d} f(x) g(x) W_\bk(x) \d x, \qquad 
            b_\k = \frac{\Gamma(|\bk|+d+1)}{\prod_{i=1}^{d+1} \Gamma(\k_i+1)},
$$
where $b_\bk$ is the normalized constant of $W_\bk$, so that $\la 1, 1\ra_\k = 1$. Let $\CV_n^d(W_{\bk})$ be
the space of orthogonal polynomials of degree $n$ with respect to $W_{\bk}$. It is well-known that 
$$
  \dim \CV_n^d(W_\bk) = \binom{n+d-1}{n}. 
$$
An orthogonal basis of $\CV_n^d(W_\bk)$ can be given in terms of the standard Jacobi polynomials $P_n^{(\a,\b)}$
\cite[Section 5.3]{DX}. Let $\proj_n^\bk f: L^2(W_\bk) \mapsto \CV_n^d(W_\bk)$ be the orthogonal projection operator. 
The Fourier orthogonal series of $f \in L^2(W_\bk)$ and its Parseval identity are given by 
$$
  f = \sum_{n=0}^\infty \proj_n^\bk f \quad \hbox{and} \quad \|f\|_{\bk,2} = \sum_{n=0}^\infty \left\| \proj_n^\bk f \right\|_{\bk,2},
\qquad \forall f\in  L^2(W_\bk),
$$
where $\|\cdot\|_{\bk,2} = \sqrt{\la f, f \ra_\bk}$ denotes the norm for $L^2(W_\bk)$. 

\subsection{Spectral operator and Bernstein inequality}
Orthogonal polynomials in the space $\CV_n^d(W_\bk)$ are eigenfunctions of  the spectral operator 
$\CD_\bk$ define in \eqref{eq:Dk} with an eigenvalue depending on $n$; more precisely
 \cite[Section 5.3]{DX}, 
$$
   \CD_\bk u = - n (n+|\bk| + d) u, \qquad u \in \CV_n^d(W_\bk).
$$
For convenience, we denote by $\partial_j$ the $j$-th partial derivative and we further define 
$$
   \partial_{i,j} := \partial_i - \partial_j, \qquad 1 \le i \ne j \le d.
$$
The operator $\CD_\bk$ can be written more symmetrically as 
\begin{equation} \label{eq:self-adj}
   \CD_\bk = \frac{1}{W_\bk} \left[\sum_{i=1}^d \partial_i \big( x_i (1-|x|) W_\bk(x) \partial_i \big) + 
     \sum_{1\le i<j\le d} \partial_{i,j} \big(x_i  x_j W_\bk(x) \partial_{i,j} \big)\right],
\end{equation}
where $\partial_{i,j} = \partial_i - \partial_j$. The weight function $W_\bk(x)$ is invariant under the simultaneous 
permutation of $(x_1,\ldots, x_d, 1-|x|)$ and $(\k_1,\k_2,\ldots,\k_{d+1})$. This symmetry is preserved in the 
decomposition of $\CD_\bk$ since the right-hand side of \eqref{eq:self-adj} is invariant under the simultaneous 
permutation, as can be easily verified. 

Using integration by parts, the decomposition in \eqref{eq:self-adj} leads immediately to
\begin{align} \label{eq:adj-int}
  - \int_{\triangle^d} \CD_\bk f(x) g(x) W_\bk(x) \d x 
     \, & = \sum_{i=1}^d \int_{\triangle^d}  x_i (1-|x|) \partial_i f(x) \partial_i g(x) W_\bk(x) \d x \\
     &  +  \sum_{1\le i<j\le d} \int_{\triangle^d} x_i x_j \partial_{i,j} f(x)  \partial_{i,j} g(x)  W_\bk(x) \d x, \notag
\end{align}
which implies, in particular, that $\CD_\bk$ is self-adjoint. As an application of the identity \eqref{eq:adj-int}, we
can obtain a sharp Bernstein inequality in the $L^2$ norm. Throughout this paper, we call an inequality sharp if 
equality holds for some polynomials. 

Let $\Pi_n^d$ be the space of polynomials of degree at most $n$ in $d$ variables. 

\begin{thm} \label{thm:B-ineq0}
Let $d \ge 2$, $n = 0,1,2,\ldots$ and $f \in \Pi_n^d$. Then
\begin{equation}\label{eq:B-ineq0}
 \sum_{i=1}^d \left \| \sqrt{x_i (1-|x|)} \partial_i f \right \|_{\bk,2}^2 + 
  \sum_{1\le i<j\le d} \left \|\sqrt{ x_i x_j} \partial_{i,j} f \right \|_{\bk,2}^2 \le n(n+|\bk|+d) \|f \|_{\bk,2}^2
\end{equation}
and the equality holds if and only if $f \in \CV_n^d(W_\bk)$.  \footnote{We thank Andras Kro\'o for pointing out this
characterization for equality.} 
Furthermore, the following two inequalities are also sharp,
\begin{align}
 \sum_{i=1}^d \left \| \sqrt{x_i (1-|x|)} \partial_i f \right \|_{\bk,2}^2 \le n(n+|\bk|+d) \|f \|_{\bk,2}^2, \label{eq:B-ineq0a}
\end{align}
and, for $1 \le \ell \le d$,  
\begin{align}
\left \| \sqrt{x_\ell (1-|x|)} \partial_\ell f \right \|_{\bk,2}^2+ 
    \sum_{\substack{j=1 \\ j \ne \ell}}^d \left \|\sqrt{x_\ell x_j} \partial_{\ell,j} f \right \|_{\bk,2}^2 \le n(n+|\bk|+d) \|f \|_{\bk,2}^2. \label{eq:B-ineq0b}
\end{align}
\end{thm}

\begin{proof}
Let $\l_n^\bk = n(n+|\bk|+d)$. Since $f$ is a polynomial of degree $n$, we have
$$
  f = \sum_{j=0}^n \proj_j^\bk f  \quad \hbox{and} \quad  \CD_\bk f = \sum_{j=0}^n \l_j^\bk \proj_j^{\bk} f.
$$
In particular, by the orthogonality and the Parseval identity, we obtain
$$
  \left\|\CD_\bk f(x) \right\|_{\bk,2}^2 = \sum_{j=0}^n (\l_j^\bk)^2  \left\|\proj_j^{\bk} f \right\|_{\bk,2}^2
    \le  (\l_n^{\bk} )^2 \sum_{j=0}^n  \left\|\proj_j^{\bk} f \right\|_{\bk,2}^2 =( \l_n^{\bk})^2 \|f\|_{\bk,2}^2.
$$
Consequently, by the Cauchy-Schwarz inequality, we deduce 
\begin{equation} \label{eq:int_Dfg}
   \left|c_\bk \int_{\triangle^d} \CD_\bk f(x)\cdot g(x) W_\bk(x) \d x \right| \le 
     \left\|\CD_\bk f(x) \right\|_{\bk,2} \cdot \|f\|_{\bk,2} \le \l_n^{\bk} \|f\|_{\bk,2}^2.
\end{equation}
Setting $g = f$ in \eqref{eq:adj-int} and applying the above inequality, we have proved \eqref{eq:B-ineq0}, whereas
\eqref{eq:B-ineq0a} and \eqref{eq:B-ineq0b} are immediate consequences of \eqref{eq:B-ineq0}. To see that \eqref{eq:B-ineq0} 
is sharp, we notice that it becomes equality for any orthogonal polynomial in $\CV_n^d(W_\bk)$ by the spectral
property $\CD_\bk f  = \l_n^\bk f$ and, moreover, the above proof also shows that the reversed inequality \eqref{eq:B-ineq0}
holds if $f\in \CV_n^d(W_\bk)$. Furthermore, we consider the Jacobi polynomial 
\begin{equation} \label{eq:Pe1}
  P_{e_1}^\bk(x) = P_n^{(|\bk|-\k_1 + d-1,\k_1)}(2 x_1-1),
\end{equation}
which is the orthogonal polynomial of degree $n$ in $\CV_n^d(W_{\bk})$ by setting $\a_1 = n$ and $\a_2 \ldots = \a_{d} = 0$  
for the orthogonal polynomial $P_{\a}$ given in \cite[Prop. 5.3.1]{DX}. Choosing $f = g =  P_{e_1}^\bk(x)$ in \eqref{eq:adj-int}
shows that \eqref{eq:B-ineq0b} is sharp for $\ell = 1$. The case of $\ell \ne 1$ follows from the simultaneous permutation 
of variables and parameters. Furthermore, making a change of variables $(x_1,\ldots, x_d) \mapsto 
(1-|x|, x_2,\ldots,x_{d-1})$ and simultaneously for $\bk$, we see that the polynomial 
\begin{equation} \label{eq:Re1}
  R_{e_1}^\bk(x) = P_{e_1}^\bk(1-|x|, x_2,\ldots, x_d) = P_n^{(|\bk|-\k_{d+1}+ d-1,\k_{d+1})}(1- 2 |x|)
\end{equation}
is also an element in $\CV_n^d(W_\bk)$. Setting $f = g = R_n$ in \eqref{eq:adj-int} shows that equality holds in \eqref{eq:B-ineq0b}. 
Moreover, we evidently have $\partial_{i,j} R_n(x) = 0$. Hence, the identity for \eqref{eq:B-ineq0a} holds as well for $R_{e_1}^\bk$. 
\end{proof}

As an immediate consequence of the theorem, we obtain the Bernstein inequalities 
\begin{align}
 \left \| \sqrt{x_i (1-|x|)} \partial_i f \right \|_{\bk,2} & \le  \sqrt{n(n+|\bk|+d)} \|f \|_{\bk,2}, \quad 1 \le i \le d, \label{eq:B-ineq0c}\\
  \left \| \sqrt{x_i x_j} \partial_{i,j} f \right \|_{\bk,2} & \le \sqrt{n(n+|\bk|+d)} \|f \|_{\bk,2}, \quad 1 \le i, j \le d, \label{eq:B-ineq0d}
\end{align}
for any polynomial $f$ of degree at most $n$. These inequalities are classical and, apart from the explicit constant 
on the right-hand side, hold for the weighted $L^p$ norm (\cite{BX, Dit, DT}). 

\subsection{New decomposition of spectral operator and Bernstein inequality}\label{sect:2.2}
In this subsection, we present another type of decomposition of the spectral operator $\CD_\bk$, which is characteristically 
different from \eqref{eq:self-adj} and is somewhat surprising. It leads to a number of new Bernstein inequalities, including 
those on the triangle mentioned in the introduction. 

\begin{thm} \label{thm:CD_k2}
For $d \ge 2$, the spectral operator $\CD_\bk$ on $\triangle^d$ satisfies 
\begin{align} \label{eq:CD_k2}
  \CD_\bk = \frac{1}{W_\bk(x)}&  \left[\frac{1}{|x|^d} \la x ,\nabla\ra \left(|x|^{d-1}(1-|x|) W_\bk(x)  \la x ,\nabla\ra \right) \right. \\
                  &  \qquad \qquad\quad  \left. +    \frac{1}{|x|} \sum_{1 \le i < j \le d} \partial_{i,j}\left(  x_i x_j W_\bk(x) \partial_{i,j}\right) \right]. \notag
\end{align}
\end{thm} 

\begin{proof}
Let $e = (0,\ldots,0,1) \in \RR^{d+1}$. We start with the identity 
\begin{align*}
 L : = & \la x ,\nabla\ra  W_{\bk+e}  \la x ,\nabla\ra = \sum_{i=1}^d x_i \partial_i \sum_{j=1}^d W_{\bk+e}(x) x_j \partial_j \\
  = \, & |x| \sum_{j=1}^d \partial_j (W_{\bk+e}(x) x_j)\partial_j 
        -  \sum_{i=1}^d x_i \sum_{j=1}^d \left[ \partial_j\big(W_{\bk+e} (x)x_j\partial_j  \big)
          - \partial_i \big(W_{\bk+e}(x)x_j  \partial_j\big) \right]\\
   =\, &  |x| \sum_{i=1}^d \partial_i \big(W_{\bk+e}(x)x_i \partial_i\big) +  
            \sum_{i=1}^d x_i \sum_{j=1}^d (\partial_i-\partial_j) \big(W_{\bk+e}(x)  x_j\partial_j\big) : = L_1 + L_2.
\end{align*}
Using $x_k \partial_k f(x) = \partial_k x_k f(x) - f(x)$, the second term $L_2$ on the right-hand side can be written as 
\begin{align*}
  L_2 \, & = 
  \sum_{j=1}^d \Bigg[ \sum_{\substack{i=1\\i\ne j}}^d (\partial_i-\partial_j) \big(x_i x_j W_{\bk+e} (x)\partial_j\big) - (d-1) W_{\bk+e} (x) x_j \partial_j \Bigg]\\
 & =  - \sum_{1 \le i < j \le d}\partial_{i,j}  \big( x_i x_j W_{\bk+e}(x) \big) \partial_{i,j} 
    - (d-1) |x| W_{\bk+e}(x) \la x, \nabla \ra,
\end{align*}
where the second step, using $\partial_{i,j} = \partial_i - \partial_j$, follows from the identity 
$$
\sum_{j=1}^d  \sum_{\substack{i=1\\i\ne j}}^d A_{i,j} = \sum_{1\le i<j\le d} (A_{i,j} + A_{j,i}). 
$$
Hence, using $W_{\bk+e}(x) = (1-|x|)W_\k(x)$ and $\partial_{i,j} \phi(|x|) = 0$, it follows that 
\begin{align*}
L= \la x ,\nabla\ra W_{\bk+e}(x)  \la x ,\nabla\ra \,  = & |x| \sum_{i=1}^d \partial_i \big(x_i W_{\bk+e} (x)\big)\partial_i 
        - (d-1) |x| W_{\bk+e}(x) \la x, \nabla \ra \\ 
      & - (1-|x|) \sum_{1 \le i < j \le d}\partial_{i,j}  \big( x_i x_j W_\bk(x) \big) \partial_{i,j}.
\end{align*}
Rearranging the terms, the above identity yields 
\begin{align*}
  |x| \CD_{\bk}\, & =  |x|   \sum_{i=1}^d \partial_i \big(x_i W_{\bk+e} (x)\big)\partial_i 
    + |x|\sum_{1 \le i < j \le d}\partial_{i,j}  \big( x_i x_j W_\bk(x) \big) \partial_{i,j}\\ 
  & =  \la x ,\nabla\ra  \big(W_{\bk+e}(x)  \la x ,\nabla\ra\big) + (d-1) W_{\bk+ e}(x) \la x, \nabla \ra 
    + \sum_{1 \le i < j \le d}\partial_{i,j}  \big( x_i x_j W_\bk(x) \big) \partial_{i,j} \\
   & =  \frac{1}{|x|^{d-1}} \la x ,\nabla\ra \big( |x|^{d-1} W_{\bk+e}(x)  \la x ,\nabla\ra\big)
   + \sum_{1 \le i < j \le d}\partial_{i,j}  \big( x_i x_j W_\bk(x) \big) \partial_{i,j}, 
\end{align*}
where the second equality follows from applying the identity
$$
  \la x ,\nabla\ra \big( |x|^{d-1} G(x) \big) = |x|^{d-1} \big[ \la x ,\nabla\ra G(x) + (d-1) G(x)\big] 
$$
with $G(x) = W_{\bk+ e}(x) \la x, \nabla \ra$. This completes the proof. 
\end{proof}

It is worth mentioning that while the right-hand side of \eqref{eq:self-adj} is invariant under the simultaneous permutation of 
$\bk$ and $(x, 1-|x|)$, the right-hand side of \eqref{eq:CD_k2} is not. Indeed, we have the following: 

\begin{cor} \label{cor:CD_k3}
Let $\ell$ be an integer $1 \le \ell \le d$. Then the spectral operator $\CD_\bk$ satisfies 
\begin{align*} 
  \CD_\bk = & \frac{1}{W_\bk(x)}  \left[\frac{1}{(1-x_\ell)^d} (\la x ,\nabla\ra - \partial_\ell) 
      \left (x_\ell (1-x_\ell)^{d-1}W_\bk(x) (\la x ,\nabla\ra - \partial_\ell)\right) \right. \\
                  &  + \frac{1}{1-x_\ell} 
                   \sum_{\substack{1 \le i \le d \\ i \ne \ell}} \partial_i \big(  x_i (1-|x|) W_\bk(x) \partial_i \big)  
                   + \frac{1}{1-x_\ell} 
                   \sum_{\substack{1 \le i < j \le d \\ i \ne \ell, j \ne \ell}} 
                      \partial_{i,j}\left(  x_i x_j W_\bk(x) \partial_{i,j}\right) \bigg]. \notag
\end{align*}
\end{cor}

\begin{proof}
It suffices to prove the case when $\ell = 1$. We make a simultaneous change of variables 
$(x_1,\ldots, x_d) = (y_2, y_3,\ldots, y_d, 1-|y|)$ and $\bk \mapsto (\k_2, \ldots, \k_{d+1}, \k_1)$. Then 
$W_\bk(x)$ is unchanged and $\partial_{x_j} = \partial_{y_{j+1}} - \partial_{y_1}$ for $1 \le j \le d-1$ and 
$\partial_{x_d} = - \partial_{y_1}$. Hence, it follows readily that $\la x, \nabla_x\ra = \la y,\nabla_y\ra - \partial_{y_1}$; 
moreover, $\partial_{x_i} - \partial_{x_j} = \partial_{y_{i+1}}-\partial_{y_{j+1}}$ for $1 \le i <  j \le d-1$ and 
$\partial_{x_i} - \partial_{x_d} = \partial_{y_{i+1}}$ for $1 \le i \le d-1$. It follows that the right-hand side 
of \eqref{eq:CD_k2} becomes the stated new decomposition in $y$ variables for $\ell = 1$. 
\end{proof} 

As an application of the new decomposition of $\CD_\bk$, we obtain alternative expressions for the
integral in \eqref{eq:adj-int}. 

\begin{thm} 
Let $f$ and $g$ be functions in $C^2(\triangle^d)$. Then 
\begin{align} \label{eq:self-adj2}
  -\int_{\triangle^d} \CD_\bk f(x) \cdot g(x) W_\bk(x) \d x 
     \,= &  \int_{\triangle^d} \la x ,\nabla \ra f(x) \cdot \la x ,\nabla \ra g(x) (1-|x|)W_\bk(x) \frac{\d x}{|x|} \\
      &  +  \sum_{1\le i<j\le d} \int_{\triangle^d} \partial_{i,j} f(x)  \partial_{i,j} g(x)  x_i x_j W_\bk(x) \frac{\d x}{|x|}. \notag
\end{align}
Moreover, for $1 \le \ell \le d$, 
\begin{align} \label{eq:self-adj3}
  -\int_{\triangle^d}  \CD_\bk f(x)& \cdot g(x)  W_\bk(x)  \d x = 
      \sum_{\substack{i=1 \\ i \ne \ell}}^d \int_{\triangle^d} \partial_i f(x) \cdot  \partial_i g(x)  x_i (1-|x|) \frac{\d x}{1-x_\ell} \\ 
       + & \sum_{\substack{1 \le i < j \le d \\ i \ne \ell, \, j \ne \ell}}  
          \int_{\triangle^d} \partial_{i,j} f(x)  \partial_{i,j} g(x)  x_i x_j W_\bk(x) \frac{\d x}{1-x_\ell} \notag \\
       + & \int_{\triangle^d} (\la x ,\nabla \ra- \partial_\ell) f(x) \cdot (\la x ,\nabla \ra g(x)- \partial_\ell) 
         g(x) x_\ell W_\bk(x) \frac{\d x}{1-x_\ell} \notag.
\end{align}
\end{thm}
 
\begin{proof}
We prove \eqref{eq:self-adj2} via integration by parts of the expression of $\CD_\bk$ in \eqref{eq:CD_k2}. 
For the first term in the right-hand side of \eqref{eq:CD_k2}, we use the following identity 
$$
  \sum_{i=1}^d \partial_i \left[ \frac{x_i}{|x|^d} g (x)\right] = \sum_{i=1}^d \left( \frac{1}{|x|^d} - \frac{x_i}{|x|^{d+1}} \right) g(x)
     +\frac{1}{|x|^d}\la x,\nabla \ra g(x) = \frac{1}{|x|^d}\la x,\nabla \ra g(x),
$$ 
which implies immediately that the integral by parts gives
\begin{align*}
 - \int_{\triangle^d} & \frac{1}{|x|^d} \la x ,\nabla\ra \big(|x|^{d-1}W_{\bk+e}(x)  \la x ,\nabla\ra f(x)\big) \cdot g (x) \d x \\
 & = \sum_{i=1}^d  \int_{\triangle^d} |x|^{d-1}W_{\bk+e}(x)  \la x ,\nabla\ra f(x) \partial_i \left[ \frac{x_i}{|x|^d} g (x)\right] \d x \\
 & = \int_{\triangle^d} W_{\bk+e}(x)  \la x ,\nabla\ra f(x)  \la x, \nabla \ra g (x) \frac{\d x}{|x|}. 
\end{align*}
For the second term in the right-hand side of \eqref{eq:CD_k2}, the factor $\f 1 {|x|}$ can be put inside the derivatives,  
using $\partial_{i,j} \phi(|x|) = 0$, so that the integration by parts can be carried out straightforwardly.
The proof of \eqref{eq:self-adj2} follows likewise via integration by parts of the expression of $\CD_\bk$ in the
Corollary \ref{cor:CD_k3}. 
\end{proof} 
 
As an immediate application of the new integral expressions in the above theorem, we can state the
following $L^2$ Bernstein inequality. 
 
\begin{cor} 
Let $d \ge 2$, $n = 0,1,2,\ldots$ and $f \in \Pi_n^d$. Then
\begin{align} \label{eq:B-ineq1}
     \left \|\frac{\sqrt{1-|x|}}{\sqrt{|x|}}  \la x ,\nabla \ra f\right \|_{\bk,2}^2
       +  \sum_{1\le i<j\le d} \left \| \frac{\sqrt{x_i x_j}}{\sqrt{|x|}} \partial_{i,j} f \right\|_{\bk,2}^2 \le 
            n(n+|\bk|+d) \|f\|_{\bk,2}^2,
\end{align}
and the equality holds if and only if $f \in \CV_n^d(W_\bk)$. Furthermore, the following inequality is sharp, 
\begin{align} \label{eq:B-ineq1a}
\left \|\frac{\sqrt{1-|x|}}{\sqrt{|x|}}  \la x ,\nabla \ra f\right \|_{\bk,2} \le \sqrt{n(n+|\bk|+d)}  \|f\|_{\bk,2}. 
\end{align}
\end{cor}
 
\begin{proof} 
These inequalities follow from the Parseval identity and \eqref{eq:int_Dfg}. Equality is attained in \eqref{eq:B-ineq1} for any 
$ f\in \CV_n^d(W_\bk)$ and in \eqref{eq:B-ineq1a} for $R_{e_1}^\bk$ in \eqref{eq:Re1}.
\end{proof}

The above corollary is derived using \eqref{eq:self-adj2}. We could also derive another set of inequalities using
\eqref{eq:self-adj3}, which we state in the following. 

\begin{cor} 
Let $d \ge 2$, $n = 0,1,2,\ldots$ and $f \in \Pi_n^d$. Then, for $1 \le \ell \le d$, 
\begin{align} \label{eq:B-ineq2}
    & \left \|\frac{\sqrt{x_\ell}}{\sqrt{1-x_\ell}}  (\la x ,\nabla \ra - \partial_\ell) f\right \|_{\bk,2}^2  + 
      \sum_{\substack{i=1 \\ i \ne \ell}}^d  \left \| \frac{\sqrt{x_i (1-|x|)}}{\sqrt{1- x_\ell}} \partial_i f \right\|_{\bk,2}^2 \\
    & \qquad\qquad\qquad\qquad
      +  \sum_{\substack{1 \le i < j \le d \\ i \ne \ell, j \ne \ell}}  \left \| \frac{\sqrt{x_i x_i}}{(1-x_\ell)} \partial_{i,j} f \right\|_{\bk,2}^2 \le 
            n(n+|\bk|+d) \|f\|_{\bk,2}^2, \notag
\end{align}
which is sharp and so is the inequality
\begin{align} \label{eq:B-ineq2a}
\left \|\frac{\sqrt{x_\ell}}{\sqrt{1-x_\ell}}  (\la x ,\nabla \ra - \partial_\ell) f\right \|_{\bk,2} \le \sqrt{n(n+|\bk|+d)}  \|f\|_{\bk,2}. 
\end{align}
\end{cor}

The identity in \eqref{eq:B-ineq2a} is attained for $P_{e_1}$ in \eqref{eq:Pe1}. For comparison, we state two more
inequalities that are consequences of \eqref{eq:B-ineq1} and \eqref{eq:B-ineq2}. 

\begin{cor}
Let $d \ge 2$, $n = 0,1,2,\ldots$ and $f \in \Pi_n^d$. Then, for $1 \le i \le d$,
\begin{align} \label{eq:B-ineq3}
   \max_{\substack{1 \le \ell \le d \\ \ell \ne i}} \left \| \frac{\sqrt{x_i  (1-|x|)}}{\sqrt{1- x_\ell}} \partial_i f \right\|_{\bk,2} 
        \le  \sqrt{n(n+|\bk|+d)}\|f\|_{\bk,2}.
\end{align}
Moreover, for $1 \le i,j \le d$ and $x_{d+1} = 1- |x|$, 
\begin{align} \label{eq:B-ineq4}
  \max_{\substack{1 \le \ell \le d +1\\ \ell \ne i, \, \ell \ne j}} 
  \left \| \frac{\sqrt{x_i x_j}}{\sqrt{1- x_\ell}} \partial_{i,j} f \right\|_{\bk,2} \le  \sqrt{n(n+|\bk|+d)}\|f\|_{\bk,2}.
\end{align}
\end{cor}

These inequalities are stronger than the classical inequalities in \eqref{eq:B-ineq0c} and \eqref{eq:B-ineq0d}. 
For $d = 2$, they are the Bernstein inequalities on the triangle that appeared in \cite{X23}.

\section{$L^p$ Bernstein inequalities for doubling weight}
\setcounter{equation}{0}

We prove Bernstein inequalities in the $L^p$ norm with respect to a doubling weight. The main result will
be stated and discussed in the first subsection. The proof follows the approach in \cite{X23}, which allows us 
to handle the doubling weight. In the second subsection, we verify that  $(\triangle^d, W_{\bk}, \sd)$ is
indeed a space of localized homogenous type. In the third subsection, we estimate the derivatives of the
localized kernel. The proof of Bernstein's inequality is given in the fouth subsection. 

\subsection{Bernstein inequalities}
\setcounter{equation}{0}

We first recall the definition of a doubling weight, which we restrict on the simplex. The distance function $\sd$ on
$\triangle^d$ is defined by
$$
  \sd_\triangle(x,y) = \arccos \left( \sqrt{x_1} \sqrt{y_1} + \cdots  \sqrt{x_d} \sqrt{y_d}+  \sqrt{1-|x|} \sqrt{1-|y|}\right). 
$$
For $x\in \triangle^d$ and $r > 0$, we denote a ball centered at $x$ with radius $r$ in $\triangle^d$ by $\sB(x,r)$ and
denote its measure with respect to a weight function $\sw$ by $\sw(B(x,r))$; then
$$
\sB(x,r):=\left \{y\in \triangle^d: \sd_\triangle(x,y) \le r\right\} \quad \hbox{and}\quad \sw(\sB(x,r)) := b_\bk \int_{\sB(x,r)} W_\bk(y) \d y.
$$  
A weight function $\sw$ is called a doubling weight on $\triangle^d$ if its satisfies the doubling condition:
 there exists a constant $c_\sw > 0$ such that
$$
  \sw(\sB(x,2 r)) \le c_\sw \, \sw(\sB(x,r)), \qquad \forall x\in \triangle^d, \,\,  0 < r \le r_0,
$$ 
where $r_0$ is the largest positive number such that $\sB(x,r) \subset \triangle^d$. For a doubling weight $\sw$, we 
denote the smallest $c_\sw$ by $c(\sw)$ and call it the {\it doubling constant}. We also call the positive number 
$\a(\sw)$ that satisfies
$$
\sup_{\sB(x,r) \subset \triangle^d} \frac{\sw(\sB(x, 2^m r))}{\sw(\sB(x,r))} \le c_w 2^{\a(\sw)} m, \quad m =0,1,2,\ldots,
$$
the doubling index. Iteration shows that $\a(\sw) \le \log_2 c(\sw)$. 

To state our Bernstein inequalities on the simplex, we introduce the following functions: for 
$x = (x_1,\ldots, x_d) \in \triangle^d$ and $1 \le i, j \le d$, 
$$
  \phi_{i,j} (x) := \frac{\sqrt{x_i} \sqrt{x_j}}{\sqrt{x_i+x_j}} \quad \hbox{and} \quad \phi_{i} (x) := \frac{\sqrt{x_i}\sqrt{1-|x|}}{\sqrt{x_i+ 1- |x|}}. 
$$
Moreover, for a doubling weight $\sw$, we let $\|\cdot\|_{\sw,p}$ denote the norm of $L^p(\triangle^d, \sw)$ for 
$1 \le p \le \infty$, where we adopt the convention that the norm becomes $\|\cdot\|_\infty$ when $p = \infty$ 
and the space becomes $C(\triangle^d)$. 

\begin{thm}\label{thm:B-Lp}
Let $\sw$ be a doubling weight on $\triangle^d$. For $n, r \in \NN_0$ and $1 \le p \le \infty$,
\begin{equation}\label{eq:B_p1}
   \left\| \phi_i^r \partial_i^r f \right\|_{\sw,p} \le c_p n^r \|f\|_{\sw,p},\quad 1 \le i \le d,
\end{equation} 
and 
\begin{equation}\label{eq:B_p2}
   \left\| \phi_{i,j}^r \partial_{i,j}^r f \right\|_{\sw,p} \le c_p n^r \|f\|_{\sw,p}, \quad 1 \le i < j \le d,
\end{equation} 
where $c_p$ is a constant depending only $\sw$, $d$ and $p$. 
\end{thm} 

These inequalities are evidently stronger than the $L^p$ version of \eqref{eq:B-ineq0c} and \eqref{eq:B-ineq0d} in
the literature \cite{BX,Dit, DT}. For a geometric perspective, let us consider the distance function $\sd_{[0,1]}(x,y)$ 
on $[0,1]$, defined by 
$$
  \sd_{[0,1]}(x,y) = \arccos \left( \sqrt{x} \sqrt{y} + \sqrt{1-x}\sqrt{1-y}\right), \qquad x, y \in [0,1].
$$
For $x \in [0,1]$, it is well-known that the distance from $x$ to the boundary of $[0,1]$ is proportional to $\sqrt{x (1-x)}$. 
Indeed, using $1-\cos \t = 2 \sin^2 \frac{\t}{2} \sim \t^2$, we have $\sd_{[0,1]}(x,0) \sim \sqrt{x}$ and 
$\sd_{[0,1]}(x,1) \sim \sqrt{1-x}$; thus
$$
    \sqrt{x(1-x)}  \sim \min \{\sd_{[0,1]}(x,0), \sd_{[0,1]}(x,1)\}.
$$ 
For $x \in \triangle^d$, we evidently have $0 \le \frac{x_i}{x_i+x_j} \le 1$, so that 
$$
   \phi_{i,j}(x) =  \min \left\{ \sd_{[0,1]}\left( \frac{x_i}{x_i+x_j},0\right),  \sd_{[0,1]} \left( \frac{x_i}{x_i+x_j}, 1\right)\right\}. 
$$
Similarly, since $0 \le x_{i}  \le 1-|x| + x_i$, we also have 
$$
 \phi_{i}(x) \sim \min \left\{  \sd_{[0,1]}\left( \frac{x_i}{1-|x|+x_i},0\right ), 
 \sd_{[0,1]}\left( \frac{x_i}{1-|x|+x_i}, 1\right ) \right\}.
$$
Note that neither $\phi_{i,j}$ or $\phi_i$ is proportional to the Euclidean distance of $x$ to the boundary of $\triangle^d$. 

As another interesting aspect of the new inequalities in Theorem \ref{thm:B-Lp}, let us compare them when $p=2$ and $r=1$
with the inequalities
derived from the spectral operator in Subsection \ref{sect:2.2}. It is easy to verify that the inequalities in \eqref{eq:B_p1} 
and \eqref{eq:B_p2} agree with those in \eqref{eq:B-ineq3} and \eqref{eq:B-ineq4} when $d = 2$, but they differ when 
$d >2$. Moreover, for $d \ge 3$, $1 \le \ell \le d$ and $x \in \triangle^d$, 
$$
   1- x_\ell \ge \sum_{j=1, j \ne \ell}^d x_j  \ge x_i+x_j, \qquad i\ne \ell, \, j \ne \ell, 
$$
and $1-x_\ell \ge x_i + 1-|x|$ for $\ell \ne i$, which shows, for example, 
$$
  \max_{\substack{1 \le \ell \le d \\ \ell \ne i}} \left \| \frac{\sqrt{x_i  (1-|x|)}}{\sqrt{1- x_\ell}} \partial_i f \right\|_{\bk,2} 
        \le  \left \| \frac{\sqrt{x_i  (1-|x|)}}{\sqrt{1- |x| + x_i}} \partial_i f \right\|_{\bk,2}.
$$ 
Consequently, if we disregard the constant $c_p$ in the right-hand side, then the inequalities \eqref{eq:B_p1} and  \eqref{eq:B_p2}
are in fact stronger than those in \eqref{eq:B-ineq3} and \eqref{eq:B-ineq4} for $d \geq 3$, as the above inequality shows, 
which is somewhat unexpected. 

The proof of Theorem \ref{thm:B-Lp} will be given in the last subsection, after we establish the necessary 
estimates for highly localized kernels on the simplex. For the weight function $W_\bk$, there could be other
proofs based on changing variables and making use of results in one variable. Our proof has the advantage
that it works for all doubling weight functions and, as a byproduct, it shows that $W_\bk$ and the simplex
$\triangle^d$ fits into the general scheme developed in \cite{X21} for localizable space of homogeneous type. 

\subsection{Highly localized kernels} 
\setcounter{equation}{0}

The projection operator $\proj_n^\bk: L^2(W_\bk, \triangle^d) \mapsto \CV_n^d(W_\bk)$ is an integral operator 
$$
   \proj_n^\bk f(x) = b_\bk  \int_{\Delta^d} P^\bk_n(x,y) f(x,y) W_\bk(y) \d y,   
$$
where $P_n^\bk(\cdot, \cdot)$ is the reproducing kernel of $\CV_n^d(W_\bk)$, which ensures $\proj_n^\bk f = f$
for all $f\in \CV_n^d(W_\bk)$. For $\k_i \ge 0$, $1\le i \le d+1$, the reproducing kernel satisfies a closed-form formula 
\cite[Theorem 5.2.4]{DX}
\begin{equation}\label{eq:add-formula}
   P_n^\bk(x,y) = a_\bk \int_{-1}^1 Z_{n}^{(|\bk| + d-\f12, -\f12)} \left(2 \xi(x,y; t)^2 -1\right) \prod_{i=1}^{d+1} (1-t_i^2)^{\k_i - \f12} \d t,
\end{equation}
where $Z_n^{(\a,\b)}$ is a constant multiple of the Jacobi polynomial
$$
  Z_n^{(\a,\b)}(t) = \frac{P_n^{(\a,\b)}(1)P_n^{(\a,\b)}(t)}{h_n^{(\a,\b)}}
$$
and $\xi(x,y,t)$ is a function defined by 
$$
  \xi(x,y; t) = \sqrt{x_1 y_1} t_1 + \cdots +\sqrt{x_d y_d} t_d + \sqrt{x_{d+1} y_{d+1}} t_{d+1},
$$
and $a_\bk$ is the normalization constant that ensures $P_0(W_\bk; x,y) =1$, given by 
$$
  a_\bk = a_{\k_1} \ldots a_{\k_{d+1}} \quad \hbox{with} \quad a_\k =   \frac{\Gamma(\k+1)}{\sqrt{\pi} \Gamma(\k+\f12)}. 
$$
The formula \eqref{eq:add-formula} holds if one or more $\k_i$ is zero under the limit. 

Let $\wh a$ be an admissible cut-off function, defined as a $C^\infty(\RR_+)$ that has support 
$[0,2]$ and satisfies $\wh a(t) =1$ for $0\le t \le 1$. We consider localized kernels defined by 
\begin{equation*} 
L_n^\bk(x,y) := \sum_{j=0}^\infty \wh a \Big(\frac{j}{n}\Big) P_j^\bk (x,y).
\end{equation*}
Using the closed-form formula of the reproducing kernel, this kernel can be written as 
\begin{equation} \label{eq:kerLn^bk}
 L_n^\bk (x,y) = a_{\bk} \int_{[-1,1]^{d+1}} L_n^{(|\bk| + d-\f12, -\f12)}\left(2 \xi(x,y;t)^2 -1\right)
       \prod_{i=1}^{d+1} (1-t_i^2)^{\k_i - \f12} \d t,
\end{equation}
where $L_n^{(\a,\b)}$ denotes the localized kernel for the Jacobi polynomials 
$$
  L_n^{(\a,\b)} (t) = \sum_{j=0}^n \wh a \Big(\frac{j}{n}\Big) Z_j^{(\a,\b)}(t), \qquad t \in [-1,1], 
$$
which is highly localized and satisfies the lemma \cite[Theorem 2.6.7]{DaiX} below.

\begin{lem} \label{lem:Ln-est-1d}
Let $\ell$ be a positive integer and let $\eta$ be a function that satisfy, $\eta\in C^{3\ell-1}(\RR)$, 
$\mathrm{supp}\, \eta \subset [0,2]$ and $\eta^{(j)} (0) = 0$ for $j = 1,2,\ldots, 3 \ell-2$. Then, 
for $\a \ge \b \ge -\f12$, $t \in [-1,1]$ and $n\in \NN$, 
\begin{equation} \label{eq:DLn(t,1)}
\left| \frac{d^m}{dt^m} L_n^{(\a,\b)}(t) \right|  \le c_{\ell,m,\a}\left\|\eta^{(3\ell-1)}\right\|_\infty 
    \frac{n^{2 \a + 2m+2}}{(1+n\sqrt{1-t})^{\ell}}, \quad m=0,1,2,\ldots. 
\end{equation}
\end{lem}

Let $B(x,r)$ be ball centered at $x$ and radius $r$ in $\triangle^d$. For $n \in \NN$, it is not difficult to see 
that, for $n = 1,2, 3, \ldots$ and $\k_i \ge 0$, $1 \le i \le d+1$, 
$$
    W_\bk (B(x,r)) = b_\bk \int_{B(x,r)} W_{\bk}(y) \d y \sim r^d  \prod_{i=1}^{d+1} (\sqrt{x_i}+r)^{2 \k_i+1}, \qquad x\in \triangle^d; 
$$
which follows, for example, from \cite[Lemma 11.3.6]{DaiX} by making a change of variables $x_j \mapsto \sqrt{x_j}$. 
For convenience, we introduce 
\begin{equation*}\label{def.WT}
W_\bk(n; x) :=  \prod_{i=1}^{d+1}  (x_i+n^{-2})^{\k_i+\f12}, \qquad x\in \triangle^d, 
\end{equation*}
so that $W_\bk (B(x,n^{-1})) \sim n^{-d} W_\bk(n;x)$. 

In \cite{X21}, highly localized kernels are considered on a space of homogeneous type $(\Omega, \sw, \sd)$, 
where $\Omega$ is a domain in $\RR^d$, $\sw$ is a doubling weight defined on $\Omega$ in terms of the distance 
function $\sd$ on $\Omega$. They are defined as kernels that satisfy three assertions, the first two describe 
localization properties of the kernel and the third one is about the weight function. For the space 
$(\triangle^d, W_\bk, \sd_\triangle)$, the third property states
\begin{align}\label{eq:as3}
 \int_{\triangle^d} \f{ n^{d} W_\bk(y) dy}{W_{\bk} (n; y) (1+n\sd_\triangle (x, y))^{\g }} 
     \leq c_\bk, 
\end{align}
which will be a consequence of Lemma \ref{lem:assertion3} below. The first localization property is the 
following proposition in \cite[Theorem 7.1]{IPX}. 

\begin{prop} \label{lem:local_est1}
Let $\wh a$ be an admissible cutoff function. Then for any $\g > 0$ there exists a constant 
$c_\g$ depending on $\g$, $\kb$, and $d$ such that  
\begin{equation}\label{eq:local1}
|L_n^\bk (x,y)| \le c_\g \frac{ n^d}{\sqrt{W_\bk(n; x)}\sqrt{W_\bk(n; y)}\big(1 + n\sd(x, y)\big)^{\g} },
\quad x,y\in \triangle^d.
\end{equation}
\end{prop} 

The proof of this proposition relies on the estimate of \eqref{eq:DLn(t,1)} with $m=0$ and two lower bounds of
$1-\xi(x,y; t)$ (cf. \cite[p. 384]{IPX}),
$$
  1-\xi(x,y;t) \ge \f2{\pi^2} \sd_\triangle(x,y)^2 \quad \hbox{and} \quad 1-\xi(x,y;t)\ge \sum_{i=1}^{d+1} \sqrt{x_i y_i} (1-t_i).
$$
The key step in the proof of \eqref{eq:local1} is the following lemma, which we shall need later, and its proof is the 
central piece of the proof of Theorem 7.2 in \cite{IPX}; see the estimate of $J$ in the proof there. 

\begin{lem}\label{lem:intLn}
For $\k_i \ge 0$ and $\g \ge 2|\bk|$, 
\begin{align*}
\int_{[-1,1]^{d+1}}& \frac{1}{ [ 1+ n \sqrt{1-\xi(x,y;t)} ]^\g}
\prod_{i=1}^{d+1} (1-t_i^2)^{\k_i-\f12} dt  \\
    & \qquad \le c_\g  \frac{ n^{-2 |\bk|-d-1}}{\sqrt{W_\bk(n; x)}\sqrt{W_\bk(n; y)}(1+n\sd_\triangle(x, y))^{\g - 2|\bk|}}.
\end{align*}
\end{lem}


The second localization property of the kernel on the simplex is the inequality below. 

\begin{prop} \label{lem:local_est2}
For $0 < \delta \le \delta_0$ with some $\delta_0<1$ and $z \in B(x, \frac{\delta}{n})$, 
\begin{equation}\label{eq:local2}
   \left|L_n^\bk(x,y) - L_n^\bk(z,y)\right| \le c_\g \frac{n^{d+1} \sd_\triangle(x,z)} 
        {\sqrt{W_\bk(n; x)}\sqrt{W_\bk(n; y)}\big(1 + n \sd_\triangle(x, y)\big)^{\g-2|\bk|-2}}.
\end{equation}
\end{prop}

\begin{proof}
Denote the left-hand side of \eqref{eq:local2} by $K$. Let $\partial L(u) = L'(u)$ and let $I_t$ denote the interval 
that has the endpoints $\xi(x,y;t)$ and $\xi(z,y;t)$. Then, by \eqref{eq:kerLn^bk},
\begin{align} \label{eq:Ln-Ln1}
K \le 2 a_\bk \int_{[-1,1]^{d+1}} L_n((x,y),(z,y);t) \left |\xi(x,y;t)^2 - \xi(z,y;t)^2\right| \prod_{i=1}^{d+1} (1-t_i^2)^{\k_i - \f12} \d t, 
\end{align}
where 
\begin{align*}
 L_n((x,y),(z,y);t): = \big\| \partial L_n^{(|\bk|+d-\f12,-\f12)} \big(2(\cdot)^2-1\big)\big\|_{L^\infty(I_t)}.
 \end{align*}
Since $\max_{r\in I_v} |1+n \sqrt{1- r}|^{-\g}$ is attained at one of the endpoints of the interval, it follows from 
\eqref{eq:DLn(t,1)} with $m =1$ that 
 \begin{align*}
 L_n((x,y),(z,y);t) \le c \left[  \frac{n^{2 |\bk|+2d + 3}}{\big(1+n\sqrt{1-\xi(x,y;t)^2} \big)^{\g}} +  
     \frac{n^{2 |\bk|+2d + 3}}{\big(1+n\sqrt{1-\xi(z,y;t)^2}\big)^{\g}} \right].
\end{align*}
Using $|\xi(x,y; t)| \le 1$, we obtain $|\xi(x,y;t)^2 - \xi(z,y;t)^2| \le 2 |\xi(x,y;t)- \xi(z,y; t)|$. Since $\xi(x,y;1) = \cos \sd(x,y)$, we 
can write 
\begin{align*}
 \xi(x,y;t)- \xi(z,y;t) \,& =  \sum_{i=1}^{d+1} \sqrt{x_i y_i} - \sum_{i=1}^{d+1} \sqrt{z_i y_i} - \sum_{i=1}^{d+1} \sqrt{x_i y_i} (1-t_i) 
    + \sum_{i=1}^{d+1} \sqrt{z_i y_i} (1-t_i) \\
    & =  \cos \sd_\triangle(x,y) - \cos \sd_\triangle(z,y) + \sum_{i=1}^{d+1} \left(\sqrt{x_i}  - \sqrt{z_i}\right) \sqrt{y_i}(1-t_i). 
\end{align*}
From the elementary identity
\begin{align*}
  \cos \t-\cos \phi= 2 \sin \frac{\t - \phi}{2} \sin \frac{\t+\phi}2 
   = 2 \sin \frac{\t - \phi}{2} \left( 2 \sin \frac{\t}{2} + \sin \frac{|\t-\phi|}{2}\right),
\end{align*} 
it follows readily that $ |\cos \t -  \cos \phi| \le |\t - \phi| \left ( |\t| + \tfrac 12 |\t - \phi| \right)$. Hence, using the triangle 
inequality for the distance $\sd(\cdot,\cdot)$, we obtain  
$$
   |\cos  \sd_\triangle(x,y) - \cos  \sd_\triangle(z,y)| \le \sd_\triangle(x,z) 
       \left ( \sd_\triangle(x,z) +  \sd_\triangle(v,y)\right) \quad \hbox{with $v=x$ or $v=z$}.
$$
For $1 \le i \le d+1$, the Cauchy-Schwarz inequality implies that (\cite[(7.5)]{IPX})
$$
   \left| \sqrt{x_j} - \sqrt{y_j} \right| \le  \sd_\triangle(x,y), \quad x, y \in \triangle^d.
$$
Consequently, we conclude that 
$$
 |\xi(x,y;t)- \xi(z,y;t)| \le   \sd_\triangle(x,z) \left[  \sd_\triangle(x,z) + \sd_\triangle(v,y)  + \sum_{i=1}^{d+1} \sqrt{y_i}(1-t_i) \right].
$$
Using these estimates on the right-hand side of \eqref{eq:Ln-Ln1}, we obtain
\begin{align*}
K \le c\,  \sd_\triangle(x,z) \int_{[-1,1]^{d+1}} & \left[  \frac{n^{2 |\bk|+2d + 3}}{\big(1+n\sqrt{1-\xi(x,y;t)^2} \big)^{\g}} +  
     \frac{n^{2 |\bk|+2d + 3}}{\big(1+n\sqrt{1-\xi(z,y;t)^2}\big)^{\g}} \right] \\
    & \times \left(  \sd_\triangle(v,y) + \sd_\triangle(x,z)  + \sum_{j=1}^{d+1} \sqrt{y_j}(1-t_j) \right) \prod_{i=1}^{d+1} (1-t_i^2)^{\k_i - \f12} \d t.
\end{align*}
We write the righthand side as a sum according to the last bracket, so that 
$$
   K \le c\,  \sd_\triangle(x,z)\left(K_0 + \sum_{j=1}^{d+1} K_j \right).
$$
For $z \in B(x, \frac{\delta}{n})$, we have $ \sd_\triangle(x,z) + \sd_\triangle(v,y) \le c n^{-1}(1+ n \sd_\triangle(v,y))$. Hence, applying 
Lemma \ref{lem:intLn}, we obtain
\begin{align*}
  K_0\, & \le c \int_{[-1,1]^{d+1}}  \left[  \frac{n^{2 |\bk|+2d + 2}(1+ n  \sd_\triangle(x,y))}{\big(1+n\sqrt{1-\xi(x,y;t)^2} \big)^{\g}} 
     \prod_{i=1}^{d+1} (1-t_i^2)^{\k_i - \f12} \d t \right. \\ 
        & \qquad\qquad\qquad +  \left. \int_{[-1,1]^{d+1}}
           \frac{n^{2 |\bk|+2d + 2}(1+ n  \sd_\triangle(z,y))}{\big(1+n\sqrt{1-\xi(z,y;t)^2}\big)^{\g}} \right] 
   \prod_{i=1}^{d+1} (1-t_i^2)^{\k_i - \f12} \d t \\
 & \le c_\g  \frac{ n^{d+1}}{\sqrt{W_\bk(n; x)}\sqrt{W_\bk(n; y)}(1+n\sd_\triangle(x, y))^{\g - 2|\bk|-1}}.
\end{align*}
The estimate of $K_j$ also follows from Lemma \ref{lem:intLn} with $(1-t_j^2)^{\k_j-\f12}\d t_j$ replaced by
$(1-t_j^2)^{\k_j +\f12} \d t_j$; hence, let $e_j$ denotes the $j$-th coordinator vector of $\RR^{d+1}$, we obtain 
\begin{align*}
  K_j \, & \le c_\g \frac{ n^d \sqrt{y_j}} {\sqrt{W_{\bk+e_j} (n; x)}\sqrt{W_{\bk+e_j} (n; y)}(1+n \sd_\triangle(x, y))^{\g - 2|\bk|-2}} \\
            & \le c_\g \frac{ n^{d+1}} {\sqrt{W_{\bk} (n; x)}\sqrt{W_{\bk} (n; y)}(1+n \sd_\triangle(x, y))^{\g - 2|\bk|-2}},
\end{align*}
where we have used $n^{-1} \sqrt{y_j} \le (\sqrt{x_j} + n^{-1})(\sqrt{y_j} + n^{-1})$ in the second step. Putting these estimates
together proves \eqref{eq:local2}. 
\end{proof} 

With these localization properties established, we have verified that the kernels $L_n^\bk(\cdot,\cdot)$ are highly 
localized on the simplex. Thus, the space $(\triangle^d, W_\bk, \sd_\triangle)$ is a localizable space of 
homogeneous type as defined in \cite{X21}. From the general framework for such spaces, developed in \cite{X21}, 
a number of results on the simplex follow readily,  some of which will be used in the proof of Bernstein
inequalities in the last subsection. 

\subsection{Derivative of highly localized kernels}
Our proof of the Bernstein inequality requires the estimate of the derivatives of the highly localized kernel. 
Recall that $\partial_{i,j}:=\partial_{x_i}-\partial_{x_j}$. 

\begin{prop}\label{deriv-est}
Assume $\k_i \ge 0$, $1\le i \le d+1$ and $\g \ge 2|\bk|+2$. For $1\le i< j \le d+1$, 
\begin{align} \label{eq:deriv-est}
 |\partial_{i,j} L_n^\bk(x,y)|   \le  c_\g & \frac{ n^{d+1} ( \sqrt{x_i + x_j} + n^{-1})} { \sqrt{x_ix_j} \sqrt{W_{\bk} (n; x)}\sqrt{W_{\bk}(n; y)}
     (1+n\sd_\triangle(x, y))^{\g - 2|\bk|-2}},
\end{align} 
where $\partial_{i,j} = \partial_{x_i} - \partial_{x_j}$ for $1\le i, j \le d$ and $\partial_{i,d+1} = \partial_{x_i}$, 
and $x_{d+1} = 1-|x|$. 
\end{prop}

\begin{proof}
Taking derivative in the $x$-variable of the identity for $L_n^\bk (\cdot,\cdot)$ in \eqref{eq:kerLn^bk} gives
\begin{align}\label{deriv-expression}
	\partial_{i,j} L_n^\bk(x,y)= a_\bk \int_{[-1,1]^{d+1}} & \partial L_n^{(|\bk|+d-\f12,-\f12)}\left(2\xi(x,y,t)^2-1\right) \\
	       & \times 4\xi(x,y,t) \cdot \partial_{i,j}\xi(x,y,t)  \prod_{i=1}^{d+1} (1-t_i^2)^{\k_i - \f12} \d t, \notag
\end{align} 
where $ \partial f(z)= \f \d{\d z} f(z)$. Let $\g$ be a positive integer. Applying \eqref{eq:DLn(t,1)} with $m =1$ gives 
\begin{align*}
 \left|  \partial L_n^{(|\k|+d-\f12,-\f12)}(2\xi(x,y,t)^2-1) \right |
       \le c_\g \frac{n^{2 |\bk|+2d + 3}}{\big(1+n\sqrt{1-\xi(x,y;t) } \big)^{\g}}.   
\end{align*}
Moreover, by the definition of $\xi(x,y; t)$, taking derivative and writing $t_i = 1-(1-t_i)$, we obtain 
\begin{align*}
	\partial_{i,j} \xi(x,y,t)\, & = \f12\sqrt{\f{y_i}{x_i}}t_i-\f12\sqrt{\f{y_{j}}{x_{j}}}t_{j} \\
     & = \f12\Big( \f{\sqrt{y_ix_{j}}-\sqrt{x_i y_{j}}}{\sqrt{x_ix_{j}}}\Big)-\f12 \Big( \sqrt{\f{y_i}{x_i}}(1-t_i) - \sqrt{\f{y_{j}}{x_{j}}}(1-t_{j})\Big).     
\end{align*}   
Together these estimates give an upper bound for the right-hand side of \eqref{deriv-expression}, 
\begin{align*} 
  & \left |\partial_{i,j}L_n^\bk(x,y)\right| 
	\leq c \int_{[-1,1]^{d+1}} \frac{n^{2 |\bk|+2d + 3}}{\big(1+n\sqrt{1-\xi(x,y;t) } \big)^{\g}} \\
 &\qquad \quad \times  \left[ \left| \f{\sqrt{y_ix_{j}}-\sqrt{x_i y_{j}}}{\sqrt{x_ix_{j}}} \right | + 
    \Big(\sqrt{\f{y_i}{x_i}}(1-t_i) - \sqrt{\f{y_{j}}{x_{j}}}(1-t_{j})\Big) \right] \prod_{i=1}^{d+1} (1-t_i^2)^{\k_i - \f12} \d t,
\end{align*}  	
which we rearrange as 
\begin{align*}  
 & \left| 2\sqrt{x_ix_j}\partial_{i,j}L_n^\bk(x,y) \right |\leq c \int_{[-1,1]^{d+1}} \frac{n^{2 |\bk|+2d + 3}}{\big(1+n\sqrt{1-\xi(x,y;t) } \big)^{\g}}\\
  &\qquad \times  \big[   \left | \sqrt{y_i}\sqrt{x_{j}}-\sqrt{x_i} \sqrt{y_{j}}\right | + \sqrt{{y_i}{x_{j}}}(1-t_i)  +  \sqrt{{y_{j}}{x_{i}}}(1-t_{j})  \big] 
      \prod_{i=1}^{d+1} (1-t_i^2)^{\k_i - \f12} \d t \\
  &\qquad\qquad\qquad\qquad \quad =: c \,(J_1 + J_2 + J_3),
\end{align*}  
where the sum is split according to the sum of three terms in the square bracket of the integrant. Thus, we need to prove that 
each $J_i$ is bounded by the right-hand side of \eqref{eq:deriv-est} multiplied by $\sqrt{x_ix_j}$. For 
$$
J_1 =\int_{[-1,1]^{d+1}} \frac{n^{2 |\bk|+2d + 3}}{\big(1+n\sqrt{1-\xi(x,y;t) } \big)^{\g}} 
 \left| \sqrt{y_i}\sqrt{x_{j}}-\sqrt{x_i} \sqrt{y_{j}}\right |  \prod_{i=1}^{d+1} (1-t_i^2)^{\k_i - \f12} \d t,
$$ 
we further write 	
\begin{align*}
 2(\sqrt{y_i}\sqrt{x_{j}}-\sqrt{x_i} \sqrt{y_{j}})=(\sqrt{y_i}-\sqrt{x_i})(\sqrt{y_j}+\sqrt{x_j})-(\sqrt{y_j}-\sqrt{x_j})(\sqrt{y_i}+\sqrt{x_i})
\end{align*}
and apply the inequality \cite[(7.5)]{IPX} 
$$
  \left|\sqrt{y_i}-\sqrt{x_i}\right |\leq \sd_\triangle (x, y), \quad x,y \in \triangle^d, \quad 1 \le i \le d+1,
$$ 
as well as 
$$
|\sqrt{y_i}+\sqrt{x_i}|=|\sqrt{y_i}-\sqrt{x_i}+2\sqrt{x_i}|\leq |\sqrt{y_i}-\sqrt{x_i}|+2\sqrt{x_i}\leq \sd_\triangle (x,y)+2\sqrt{x_i}
$$ 
to obtain 
\begin{align*}
 2|(\sqrt{y_i}\sqrt{x_{j}}-\sqrt{x_i} \sqrt{y_{j}})|&\leq  \sd_\triangle(x, y)(\sd_\triangle(x,y)+2\sqrt{x_j})
     + \sd_\triangle(x, y)(\sd_\triangle(x,y)+2\sqrt{x_i})\\
 & =2 \sd_\triangle(x,y)(\sd_\triangle(x,y)+\sqrt{x_i}+\sqrt{x_j})
\end{align*}
and use this inequality to derive an upper bound  
\begin{align*}
 J_1  &\leq c_\gamma \int_{[-1,1]^{d+1}} \frac{n^{2 |\bk|+2d + 3}\sd_\triangle(x,y)(\sd_\triangle
       (x,y)+\sqrt{x_i}+\sqrt{x_j})}{\big(1+n\sqrt{1-\xi(x,y;t) } \big)^{\g}}  \prod_{i=1}^{d+1} (1-t_i^2)^{\k_i - \f12} \d t  \\
      &  \le c_\g \frac{ n^{d+1}(\sd_\triangle (x,y)+\sqrt{x_i}+\sqrt{x_j})} {\sqrt{W_{\bk} (n; x)}\sqrt{W_{\bk} (n; y)}
          (1+n\sd_\triangle(x, y))^{\g - 2|\bk|-1}},
 \end{align*}
 where the second inequality follows from Lemma \ref{lem:intLn}. Applying the inequality 
 $$
 \sd_\triangle (x,y)+A \le n^{-1} [1 + n \sd_\triangle(x,y)] + A  \le ( n^{-1}+A) [1 + n \sd_\triangle(x,y)]
 $$
with $A =\sqrt{x_i}+\sqrt{x_j}$ and reducing the power of $1 + n \sd_\triangle(x,y)$ in the denominator by $1$, we 
see that $J_1$ has the desired upper bound. 
 
For the estimates of $J_2$ and $J_3$, we use Lemma \ref{lem:intLn} with $(1-t_j^2)^{\k_j-\f12}\d t_j$ replaced by 
$(1-t_j^2)^{\k_j +\f12} \d t_j$. Let $e_i$ denote the $i$-th coordinator vector of $\RR^{d+1}$. Then
\begin{align*}
J_2 \, & = \int_{[-1,1]^{d+1}} \frac{n^{2 |\bk|+2d + 3}}{\big(1+n\sqrt{1-\xi(x,y;t) } \big)^{\g}} 
   \sqrt{{y_i}{x_j}}(1-t_i) \prod_{k=l}^{d+1} (1-t_l^2)^{\k_l - \f12} \d t\\
  & \le c_\g \frac{ n^d \sqrt{y_i}\sqrt{x_j}} {\sqrt{W_{\bk+e_i} (n; x)}\sqrt{W_{\bk+e_i} (n; y)}(1+n\sd_\triangle(x, y))^{\g- 2|\bk|-2}} \\
	& \le c_\g \frac{ n^{d+1}\sqrt{ x_j+x_i}} {\sqrt{W_{\bk} (n; x)}\sqrt{W_{\bk} (n; y)}(1+n\sd(x, y))^{\g - 2|\bk|-2}},
\end{align*} 
where we have used $n^{-1} \sqrt{y_i} \le (\sqrt{x_i} + n^{-1})(\sqrt{y_i} + n^{-1})$ in the second step. 
The proof for $J_3$ follows similarly, just exchanging the indexes $i$ and $j$. Putting these estimates together 
proves \eqref{eq:deriv-est}.	
\end{proof}
 
\subsection{Proof of the Bernstein inequality in Theorem \ref{thm:B-Lp}} 
We will need the maximal function $f^*_{\b,n}$ defined by 
\begin{align*}
	f^*_{\b,n}(x)=\max_{y\in\triangle^d}\frac{|f(y)|}{(1+n\sd_\triangle(x,y))^\b}.
\end{align*}
The highly localized kernel guarantees, by \cite[Corollary 2.11]{X21}, the following property, where we recall
that $c(\sw)$ and $\a(\sw)$ are doubling constant and doubling index, respectively. 
 
\begin{prop}\label{prop:f_bn}
Let $\sw$ be a doubling weight on $\triangle^d$. If  $0<p\leq \infty$, $f\in \Pi(\triangle^d)$ and $\b>\f{\a(\sw)}{p}$, 
then 
$$
\|f\|_{\sw,p}\leq \|f^*_{\b,n}\|_{\sw,p}\leq c\|f\|_{\sw, p}
$$
where $c$ depends on $c(\sw)$ and on $\b$ when $\b$ is either large or close to $\a(\sw)/p$.
\end{prop}

We will need two more lemmas. The case $p =2$ of the first lemma verifies \eqref{eq:as3}, the third assertion 
for highly localized kernels. 

\begin{lem}\label{lem:assertion3}
 Let $d\geq 2$. For $0< p < \infty$,  assume  $\g$ is large enough. Then for $x\in\triangle^d$, 
 \begin{align*}
  	 \int_{\triangle^d}\f{  W_\bk(y)   dy} {( {W_{\bk} (n; y)})^{\f p2}(1+n\sd_\triangle(x, y))^{\g p}}
	     \leq c\, n^{-d} (W_\bk(n;x)) ^{1-\f p2}.
\end{align*}
\end{lem}

\begin{proof} 
The proof is similar to \cite[Lemma 4.14]{X21}. Let $J_p(x)$  denote the left-hand side of the first inequality. 
Using the doubling property of $W_\bk$, it is sufficient to consider the case $p=2$. Thus, it suffices to show, 
for $x,y\in\triangle^d$, 
\begin{align}\label{when p=2}
J_2(x):=  \int_{\triangle^d}\f{  W_\bk(y)} {{W_{\bk} (n; y)} (1+n\sd_\triangle(x, y))^{2 \g}} \d y \leq c\, n^{-d}.
\end{align}
From the definition of $W_{\bk}(y)$ and $W_\bk(n; y)$, it follows readily that
\begin{align*} 
 \frac{W_{\bk}(y) }{W_\bk(n; y)}\leq \f{(1-|y|)^{\k_{d+1}}}{(1-|y|+n^{-2})^{\k_{d+1}+\f12}}\cdot \f1{\sqrt{y_1\cdots y_d}}
\end{align*} 
Using this inequality in $J_2$ and changing variables $y_i=u_i^2$ for $1\leq i\leq d$ in the resulted integral, so that 
$\frac{\d y}{\sqrt{y_1\cdots y_d}}=\d u$, we obtain 
 \begin{align*} 
  J_2(x) & \le c \int_{\BB^d_+}\f{(1- \|u\|^2)^{\k_{d+1}}}{{( {1-\|u\|^2}+n^{-2})^{\k_{d+1} +\f12 }
     (1+n  \d_\triangle (x, u^2))^{2 \g  }} }\d u, 
\end{align*}
where $\BB_+^d = \{x\in \RR^d: x_i \ge 0, \|x\| \le 1\}$ is the positive quadrant of the unit ball $\BB^d$. 
Let $\sd_\BB(\cdot,\cdot)$ be the distance function on $\BB^d$. For $x\in \triangle^d$, let 
$x^2 = (x_1^2,\ldots,x_d^2) \in \BB_+^d$ and also let $\sqrt{x} = (\sqrt{x_1},\ldots, \sqrt{x_d})$. 
Then $\sd_\triangle (x^2,y^2) = \sd_\BB(x,y)$; see \cite[p. 403 and 404]{DaiX}. Hence, the last
displayed inequality can be written as
$$
  J_2(x) \le  c \int_{\BB^d_+}\f{  (1- \|u\|^2)^{\k_{d+1}}}{{( \sqrt{1-\|u\|^2}+n^{-1})^{2\k_{d+1} +1 } 
       (1+n \sd_\BB(\sqrt{x}, u))^{2 \g }} }  \d u.
$$ 
Enlarging the integral domain to the entire ball $\BB^d$ shows that $J_2(x)$ is bounded by
an integral on the unit ball, which has appeared and estimated in \cite[Lemma 4.6]{PX}, from which 
the inequality \eqref{when p=2} follows readily. 
\end{proof}
 
The second lemma relies on the Marcinkiewicz-Zygmund inequality, which requires a definition. 
Let $\ve > 0$. A finite subset $\Xi \subset \triangle^d$ is called $\ve$-separated if 
$\sd_\triangle(x,y) \geq \ve$ for every pair of distinct points $x,y \in \Xi$, and such a set is called 
maximal if, in addition, $\displaystyle \triangle^d = \cup_{y\in\Xi} \sB (y,\ve)$. 

\begin{lem} \label{lem:shrink}
Let $\sw$ be a doubling weight function on $\triangle^d$. For $\delta > 0$ and $n\in \NN,$ let 
$$
\triangle_{n,\delta}^d =\{x\in\triangle^d: \tfrac\delta{n}< x_i \leq 1-\tfrac{\delta}n, 1 \le i \le d+1\},
$$ 
where $x_{d+1} = 1-|x|$. Then, for $f\in\Pi_n(\triangle^d)$, $1\leq p< \infty$, 
\begin{align}\label{shrink}
   \int_{\triangle^d}|f(x)|^p \sw(x)\d x\leq c_{\delta}\int_{\triangle_{n,\delta}^d}|f(x)|^p \sw(x)\d x.
\end{align} 
Moreover, let $\chi_{n,d}$ be the characteristic function of $\triangle_{n,\delta}^d$; then for $p=\infty$,
\begin{align}\label{shrink2}
	\|f\|_\infty\leq c \|f\chi_{n,\delta}(x)\|_\infty.
\end{align}
\end{lem}

\begin{proof}
We recall the Marcinkiewicz-Zygmund inequality on the simplex. Let $\Xi$ be a maximal 
$\ve$-separated set.  For $\ve = \f\delta n$ with $\delta>0$, the Marcinkiewicz-Zygmund inequality 
states \cite[Theorem 2.15]{X21} that, for $1\leq p < \infty$, there is a $\delta>0$ such that for $f\in\Pi_n$, 
\begin{align*} 
 \|f\|_{\sw, p}^p = \int_{\triangle^d}|f(x)|^p \sw(x)\d x \le c \sum_{y \in \Xi} \sw \left(\sB(y, \tfrac{\delta}{n})\right) 
    \min _{x \in \sB(y, \frac{\delta}{n})}|f(x)|^p .
\end{align*}
Hence, for a fixed $\delta>0$, by choosing $n$ sufficient large, we see that $f(x)=f(x)\chi_{n,\delta}$
for all $x \in \sB(y, \frac{\delta}{2n})$ and $y\in\Xi$. Consequently, for all $x\in \sB(y, \frac{\delta}{2n})$, 
 \begin{align*}
 \min _{x \in \sB(y, \frac{\delta}{n})}|f(x)|^p\leq  \min_{x \in \sB(y, \frac{\delta}{2n})}|f(x)|^p\leq |f(x)|^p\chi_{n,\delta}
\end{align*}
Using this inequality and the doubling property of $\sw$, we obtain 
\begin{align*}
 \|f\|_{\sw, p}^p & \leq c \sum_{y\in \Xi} \int_{\sB(y, \frac{\delta}{2n})} |f(x)\chi_{n,\delta}(x)|^p \d x 
     \leq c \int_{\triangle^d}|f(x)|^p\chi_{n,\delta}(x)\sw(x)\d x
\end{align*} 
for all $f\in\Pi_n$, which is the desired inequality \eqref{shrink}. The inequality \eqref{shrink2} 
follows immediately from \cite[(2.36)]{X21}. 
\end{proof}

\medskip\noindent
{\bf Proof of Theorem \ref{thm:B-Lp}}. 
Let $\partial_{i,d+1} = \partial_i$. By the definition of $L_n^\bk(x,y)$, every $f\in \Pi_n^d$ satisfies  
\begin{align*}
  f(x)=\int_{\triangle^d}f(y) L_n^\bk(x,y) W_\k(y)\d y.
\end{align*}
Taking $\partial_{i,j}$ derivative and using the maximal function $f^*_{\b,n}$, we obtain
\begin{align*}
   \left|\partial_{i,j}f(x)\right |\leq c f^*_{\b,n} (x) \int_{\triangle^d}\f{|\partial_{i,j} L_n^\bk(x,y) |}
          {(1+n\sd_\triangle (x,y))^\b}  W_\k(y)\d y.
\end{align*}
For $x \in \triangle^d_{n,\delta}$, we have $\f\delta n \le \sqrt{x_i+x_j}$, so that \eqref{eq:deriv-est} becomes
$$
 \left|\partial_{i,j} L_n^\bk(x,y) \right| \le c \f{ n^{d+1}  \sqrt{x_i+x_j}}
    {\sqrt{x_ix_j}  \sqrt{W_{\bk} (n; x)}\sqrt{W_{\bk} (n; y)}(1+n\sd_\triangle (x, y))^{\g - 2|\bk|-2}}.  
$$
Consequently, choosing $\g$ sufficiently large, it follows from Lemma \ref{lem:assertion3} with $p=1$ that 
$$
   \phi_{i,j}(x)  \int_{\triangle^d}\f{|\partial_{i,j} L_n^\bk(x,y) |}
          {(1+n\sd_\triangle (x,y))^\b}  W_\k(y)\d y \le c\,n^{d+1} n^{-d} = c\, n,
$$
which implies, in particular, that for $x\in \triangle_{n,\delta}$, 
$$
   \left | \phi_{i,j} \partial_{i,j} f(x) \right | \le c\, n f^*_{\b,n} (x). 
$$
Since $\phi_{i,j}(x) \le 1$, it is easy to see that $\phi_{i,j} \sw$ is also a 
doubling weight. Hence, recall $\chi_{n,\delta}(x)=\chi_{\triangle_{n,\delta}}(x)$ and choose $\delta > 0$
as in Lemma \ref{lem:shrink}, we obtain 
\begin{align*}
  \left\|\phi_{i,j} \partial_{i,j}f \right \|_{\sw, p}\leq c_\delta\| \phi_{i,j} \partial_{i,j} f\chi_{n,\delta} \|_{\sw,p}
   \le c\, n  \| f^*_{\b,n} \|_{\sw,p} \le c \, n \|f\|_{\sw,p},
\end{align*}
where the last step follows from Proposition \ref{prop:f_bn}. This proves Theorem \ref{thm:B-Lp} for $r =1$. 
The proof for $r > 1$ follows from iteration. Indeed, for fixed $i,j$ and $p$, denote by $\sw^*$ the 
doubling weight $\sw^*(x) = \phi_{i,j}^{(r-1)p} \sw(x)$. Then
$$
 \left \| \phi_{i,j}^r \partial_{i,j}^r f \right \|_{\sw,p} = 
   \left \| \phi_{i,j} \partial_{i,j} (\partial_{i,j}^{r-1} f )\right \|_{\sw^*,p} 
   \le c\, n \left\|\partial_{i,j}^{r-1} f\right \|_{\sw^*,p} =
    c\,n \left \|\phi_{i,j}^{r-1} \partial_{i,j}^{r-1} f \right \|_{\sw,p},
$$
which can be further iterated to complete the proof. 
\qed

\end{document}